\documentclass[letterpaper,12pt]{article}
\usepackage{amsmath, mathtools}
\usepackage{amssymb}
\usepackage{amsthm}
\usepackage{graphicx}
\usepackage{verbatim}
\usepackage{color}
\usepackage{url}
\newtheorem{theorem}{Theorem}

\newtheorem{lemma}{Lemma}

\newtheorem{prop}{Proposition}

\def\P{{\mathbb P}}     
\def\E{{\mathbb E}}     

\newcommand{\nocoal}{\mathcal{N\!C}}

\title{Species tree estimation under joint modeling of coalescence and duplication: sample complexity of quartet methods}

\author{
Max Hill\footnote{
Department of Mathematics, University of Wisconsin--Madison.} 
\and 
Brandon Legried\footnote{
Department of Mathematics, University of Wisconsin--Madison.} 
\and 
Sebastien Roch\footnote{
Department of Mathematics, University of Wisconsin--Madison. Corresponding author. Email: roch@math.wisc.edu. }
}

\date{
\today
}

\definecolor{Red}{rgb}{1,0,0}
\definecolor{Blue}{rgb}{0,0,1}

\begin{document}
	
	\maketitle
	
\begin{abstract}
We consider species tree estimation under a standard stochastic model of gene tree evolution that incorporates incomplete lineage sorting (as modeled by a coalescent process) and gene duplication and loss (as modeled by a branching process). Through a probabilistic analysis of the model, we derive sample complexity bounds for widely used quartet-based inference methods that highlight the effect of the duplication and loss rates in both subcritical and supercritical regimes.  
\end{abstract}

\section{Introduction}

Estimating phylogenies from the molecular sequences of existing species is a fundamental problem in computational biology that has been the subject of significant practical and theoretical work~\cite{SempleSteel:03,felsenstein2003inferring,gascuel2005mathematics,yang2014molecular,steelbook2016,warnow2017book,roch2019handson}. Rigorous statistical guarantees for inference methodologies often involve the probabilistic analysis of Markov models on trees. In particular, these analyses have uncovered deep connections with phase transitions in related statistical physics models~\cite{Mossel:03,mossel2003information,Mossel:04,Mossel:04a,steel2004cluster,BoChMoRo:06,roch2010toward,daskalakis2011bethe,mossel2011long,mihaescu2013learning,ane2017ornstein,roch2017likelihood,fan2018necessary,ganesh2019optimal}. 

In modern datasets, however,
phylogeny estimation is confounded by heterogeneity across the genome from processes such as incomplete lineage sorting (ILS), gene duplication and loss (GDL), lateral gene transfer (LGT), and others~\cite{Maddison1997}.  Inferred trees depicting the evolution of individual loci in the genome are referred to as \textit{gene trees}, while the tree representing the speciation history is called the \textit{species tree}.  Current sequencing technology allows phylogenetic estimates of species relationships for many genes, and a major challenge in reconstructing species trees is that gene trees often disagree for the reasons mentioned above. There is a burgeoning literature on the many ways of extracting speciation histories from collections of gene trees \cite{DegnanRosenberg:09,nakhleh2013computational,steelbook2016,warnow2017book,scornavacca:hal-02535070}.

In this phylogenomic context, the design and analysis of species tree estimation methods require the use of a variety of stochastic processes beyond Markov models on trees, including coalescent processes \cite{rannala2003bayes}, branching processes \cite{ArLaSe:09}, random subtree prune-and-regraft operations \cite{Galtier:07,linz2007horizontal,roch12lateral}, and tree mixtures  \cite{MatsenSteel:07,mossel2012mixtures}.  In fact, there is increasing realization in the phylogenetics community that ILS, GDL, LGT, etc.~should not be studied in isolation~\cite{RochWarnow:15,degnan2018hybridization,simion:hal-02535366,schrempf:hal-02535482} and, as a result, there has been a push to consider more complex models that combine many sources of uncertainty and discordance \cite{MENG200935,solislemus2016networks,rasmussen2012unified,dasarathy2015data,RochSteel:15,mossel2017tradeoff,longbranch2019,RoNuWa:19,allman2019logdet,allman2019nanuq,Li+:20,eulenstein2020unified}.  We study here a joint coalescent and branching process unifying ILS and GDL, as introduced in~\cite{rasmussen2012unified}.

Much is known about estimating species trees in the presence of ILS alone~\cite{rannala:hal-02535622}, as modeled by the multispecies coalescent (MSC) \cite{rannala2003bayes}. The latter model posits that, on a fixed species tree, gene trees evolve backwards in time on each branch according to the Kingman coalescent \cite{kingman1982coalescent} (see Section~\ref{section:background} for more details).  Bayesian approaches are a natural choice under such complex models of evolution~\cite{drummond2007beast}. However they do not scale well to large datasets and more computationally efficient procedures have been developed that combine inferred gene trees, sometimes referred to as summary methods~\cite{warnow2017book}. 

One such method is to deduce the species tree by a plurality vote across gene trees. Unfortunately, that approach is not statistically consistent, that is, it may not converge to the true species tree as the number of gene trees grows to infinity. Indeed, 
for unrooted species trees with more than four species, the most frequently occurring gene tree topology need not coincide with that of the species tree \cite{degnan2006}.  However, for every $4$-tuple of species---also referred to as quartet---and every locus, the most probable unrooted gene tree topology matches the species tree topology \cite{allman2011identifying}.  This implies in particular that the species tree topology is identifiable from the distribution of gene trees.  That is, different species tree topologies necessarily produce different gene tree distributions. Further, quartet-based algorithms for combining gene trees \cite{larget2010bucky,astral} are known to be statistically consistent.  Tight bounds on their sample complexity, that is, how many gene trees are needed to recover the true species tree with high probability, have also been established~\cite{Shekhar_2018}.

Less is known about estimating species trees in the presence of GDL alone. The model in \cite{ArLaSe:09} posits that, on a fixed species tree, the number of copies in a gene family evolves forward in time on each branch according to continuous-time branching process \cite{AtheryaNey:72} (see Section~\ref{section:background} for more details). 
Recently, the identifiability of the species tree in the presence of GDL alone was established in \cite{Legried821439} by showing that, similarly to the ILS alone case, for every quartet the most frequent unrooted gene tree topology matches that of the species tree.  As a result, quartet-based inference methods \cite{rabiee2019multi} were also shown to be statistically consistent. To date, no sample complexity results have been derived under this GDL model however. 

In this paper, we investigate the gene tree evolution model of \cite{rasmussen2012unified}, which unifies the multispecies coalescent and the branching process model of gene duplication and loss discussed above.  Given that quartet-based methods have strong guarantees under these models separately, it is natural to consider their performance under the joint model as well (see Section~\ref{section:background} for more details on the methods).  Numerical experiments in \cite{DHN-biorxiv2019,Legried821439} provide some evidence for the accuracy of certain quartet-based methods.  In \cite{eulenstein2020unified}, the authors give a proof of statistical consistency for one such method.  In our main result, we give the first known upper bounds on the sample complexity of species tree estimation methods under the joint effect of ILS and GDL.  Our proof, which highlights the somewhat counter-intuitive role played by the duplication and loss rates in the supercritical regime (see Section~\ref{section:background}), is complicated by the simultaneous forward-in-time/backward-in-time nature of the process. 

The rest of the paper is organized as follows. In Section~\ref{section:background}, after defining the model and inference methods, we state and discuss our results.  In Section~\ref{section:first-step}, we give a proof of identifiability including new quantitative estimates that play a role in our proof of sample complexity. The rest of the proof can be found in Section~\ref{section:sample}.

\section{Background and main results}
\label{section:background}

We first describe the model and then state our results formally.

\subsection{Problem and model}

Our input is a collection $\mathcal{T} = \{t_i\}_{i=1}^{k}$ of $k$ multi-labeled gene trees, given without estimation error. By multi-labeled, we mean that the same species can be associated to several leaves of a gene tree, corresponding to different copies of the same gene within the genome. Our goal is to output the unknown $n$-species tree $T$.

\paragraph{Model} We assume that the gene trees are generated under the DLCoal model~\cite{rasmussen2012unified}, a unified model of gene duplication, loss, and coalescence.  Gene trees are generated in two steps under DLCoal, which we now describe:

\begin{enumerate}
\item \textit{Locus tree: Birth-death process of gene duplication and loss with daughter edges.} The rooted $n$-species tree $T = (V,E)$ has vertices $V$ and directed edges $E$ with lengths $\{\eta_{e}\}_{e \in E}$.  Assume (to simplify) there is a single copy of each gene at the root of $T$.  Locus trees are generated by a top-down birth-death process~\cite{ArLaSe:09} within the species tree.  That is, on each edge, each gene copy independently duplicates at exponential rate $\lambda \geq 0$ and is lost at exponential rate $\mu \geq 0$.  Uniformly at random, one of the edges directly descending from the duplication is called a \textit{daughter} edge, and the other is called a \textit{mother} edge.  At speciation events, every surviving gene copy bifurcates and proceeds similarly in descendant edges.  Duplications and speciations are indicated in the locus tree by a bifurcation.  The resulting locus tree is then pruned of lost copies to give an observed rooted locus tree.  The rooted $n'$-individual locus tree $T_{n'} = (V_{n'},E_{n'},L_{n'})$ with edge lengths has labels $L_{n'}$ associated to each leaf of $T_{n'}$ from the species set $S$.

\item \textit{Gene tree: Coalescent process on a locus tree.} Gene trees are generated by a backward-in-time coalescent process~\cite{kingman1982coalescent,rannala2003bayes} within the locus tree.  The coalescent process begins with exactly one gene copy in each leaf of the locus tree.  Copies at the bottom of a directed edge in the locus tree undergo the Kingman coalescent process for a time equal to the length of the directed edge in coalescent time units, conditioned on the event that all gene copies at the bottom of any daughter edge in the locus tree necessarily coalesce underneath the top of the edge. 
Continuing along ancestral edges, this process eventually yields a gene tree.  The multi-copy gene trees $\{t_i\}_{i=1}^{k}$ are assumed independent and identically distributed.
This process is referred to as the \textit{bounded multi-species coalescent} (b-MSC) model in~\cite{rasmussen2012unified}.
\end{enumerate}

\paragraph{Species tree estimation methods} Next we describe two quartet-based species tree methods: ASTRAL-one and ASTRAL-multi~\cite{astral,DHN-biorxiv2019,rabiee2019multi}.  Both of these methods are practical variants of an intuitive idea (which in the ILS/GDL context is motivated by the results of~\cite{allman2011identifying,Legried821439}; see also Propositions~\ref{PropClaimBal} and~\ref{PropClaimCat} below): (1) for each quartet of species, find the most common topology across gene trees and (2) reconstruct an $n$-species tree that coincides with as many resulting quartet topologies as possible. 
The input is a collection of $k$ multi-labeled gene trees $\mathcal{T} = \{t_i\}_{i=1}^{k}$
Let $S$ be the set of $n$ species and $R$ be the set of $m$ labels (or gene copies).  The tree $t_i$ is labeled by the set $R_i \subset R$.  For any species tree $T$ labeled by $S$, the extended tree $T_{ext}$ labeled by $R$ is built by adding to each leaf of $T$ all gene copies corresponding to that species as a polytomy (i.e. as a vertex with degree possibly higher than $3$).  

Under ASTRAL-one, we pick one gene copy of each species uniformly at random and restrict the gene tree to these copies, producing a new gene tree $\tilde{t}_{i}$.  For any collection of gene copies $\mathcal{J} = \{a,b,c,d\}$, let $T^{\mathcal{J}}$ be the restriction of the tree $T$ to those copies.  Then the quartet score of every candidate species tree $\widetilde{T}$ is \begin{align*}
    Q_{k}\left(\widetilde{T}\right) = \sum_{i=1}^{k}\sum_{\mathcal{J} = \{a,b,c,d\} \subset R_i} \textbf{1}(\widetilde{T}_{ext}^{\mathcal{J}},\tilde{t}_{i}^{\mathcal{J}}),
\end{align*} 
where $\textbf{1}(T_1,T_2)$ is the indicator for the event that $T_1$ and $T_2$ have the same topology. ASTRAL-one selects the candidate tree $\widetilde{T}$ that maximizes that score.

ASTRAL-multi treats copies of a gene in species as multiple alleles within the species.  So, we do not replace $t_{i}$ with any restricted gene tree $\tilde{t}_{i}$. The quartet score of $T$ with respect to $\mathcal{T}$ is \begin{align}
    \nonumber Q_{k}\left(\widetilde{T}\right) = \sum_{i=1}^{k}\sum_{\mathcal{J} = \{a,b,c,d\} \subset R_i} \textbf{1}(\widetilde{T}_{ext}^{\mathcal{J}},t_i^{\mathcal{J}}),
\end{align} where $T_1^{\mathcal{J}}$ is the restriction of $T_1$ to individuals $\mathcal{J}$.   
The candidate tree $\widetilde{T}$ that maximizes this score is chosen by ASTRAL-multi.  

Both procedures can be performed in exact mode or in a default constrained mode, which restricts the number of candidate species trees to those displaying the bipartitions in the given gene trees.  

\subsection{Statement of main results}

We present a theoretical bound on the number of gene trees needed for ASTRAL-one to reconstruct the model species tree with high probability under the DLCoal model.  Similarly to the case of the MSC model~\cite{Shekhar_2018}, the sample complexity depends on the length of the shortest species tree branch in coalescent time units.  We denote this length by $f$.
However we also highlight the influence of other relevant parameters in the sample complexity: the depth of the species tree, $\Delta$; the duplication and loss rates, $\lambda$ and $\mu$.  For simplicity, we assume throughout that $\mu \neq \lambda$.
\begin{theorem}[Main result: Sample complexity of ASTRAL-one]
\label{TheoremSampleONE}
Consider a model species tree whose minimum branch length $f$ is finite and assume gene trees are generated under the DLCoal model.  Then, for any $\epsilon > 0$, there are universal positive constants $C, C'$ such that the exact version of ASTRAL-one returns the true species tree with probability at least $1 - \epsilon$ if the number of input error-free gene trees satisfies:
\begin{align}
    k
    \geq 
    C' \, 
    \frac{1}{f^2}\, 
    \frac{e^{C |\mu - \lambda|\Delta}}{\left(1 -  \frac{\lambda}{\mu} \land \frac{\mu}{\lambda} \right)^{C}}\, 
    \log \frac{n}{\epsilon}. \label{eq:complexity}
\end{align} 
\end{theorem}

Somewhat surprisingly the subcritical ($\mu > \lambda$) and supercritical ($\mu < \lambda$) regimes exhibit a similar behavior. Indeed one naturally expects a higher sample complexity in the subcritical case as $\mu/\lambda$ becomes large because the absence of any gene copy in a species becomes more likely and leads to the need for more gene trees in order to extract signal. That prediction is borne out in~\eqref{eq:complexity}. However the sample complexity similarly increases in the supercritical regime as $\lambda/\mu$ becomes large.   As the proof shows, the reasons for this behavior are different in that regime. They have to do with the fact that a large number of copies at the most recent common ancestor of a species quartet tends to produce large numbers of conflicting gene tree quartet topologies, thereby obscuring the signal.  It is an open problem whether other inference methods (perhaps not based on quartets) are less sensitive to this last phenomenon.

Our proof of Theorem~\ref{TheoremSampleONE} involves a detailed probabilistic analysis of the DLCoal model. Along the way, we prove other results of interest.
First, we show that the unrooted species tree is identifiable from the distribution of multi-labeled gene trees $\mathcal{T}$ under the DLCoal model over $T$.  Formally, we show that two distinct unrooted species trees produce different gene tree distributions.  The result is a generalization of~\cite[Theorem 1]{Legried821439}, where only GDL is considered. Theorem~\ref{TheoremDLCoalIdent} was first claimed in~\cite{eulenstein2020unified}. Our novel contribution here lies in the proofs of Propositions~\ref{PropClaimBal} and~\ref{PropClaimCat} below, which give a quantitative version of the identifiability result and play a role in deriving the sample complexity of ASTRAL-one. 
\begin{theorem}[Identifiability of species tree] 
\label{TheoremDLCoalIdent}
	Let $T$ be a model species tree with at least $n \geq 4$ leaves.  Then $T$, without its root, is identifiable from the distribution of gene trees $\mathcal{T}$ under the DLCoal model over $T$.
\end{theorem}
\noindent The identifiability result is established by showing that, for each quartet in the species tree, the most likely gene tree matches the species tree.
As in \cite{Legried821439,eulenstein2020unified}, a direct consequence of this proof is the statistical consistency of the ASTRAL-one.
\begin{theorem}[Consistency of ASTRAL-one] \label{TheoremDLCoalONE}
As the number of input gene trees tends toward infinity, the output of ASTRAL-one converges to $T$ almost surely, when run in exact mode or in its default constrained version.
\end{theorem}

We use a similar reasoning to prove the consistency of ASTRAL-multi.  This result is new.
\begin{theorem}[Consistency of ASTRAL-multi] \label{TheoremDLCoalMulti}
As the number of input gene trees tends toward infinity, the output of ASTRAL-multi converges to $T$ almost surely, when run in exact mode or in its default constrained version.
\end{theorem}

\section{First step: a proof of identifiability of the species tree under the DLCoal model}
\label{section:first-step}

The unrooted topology of a species tree is defined by its quartets (see, e.g., \cite{SempleSteel:03}).  Let $\mathcal{Q} = \{A, B, C, D\}$ and assume, without loss of generality, that $T^{\mathcal{Q}}$ has quartet topology $AB|CD$. 
Let $t$ be a 
gene tree generated under the DLCoal model on $T$ and let  $t^{\mathcal{Q}}$ be its restriction to the gene copies from the species in $\mathcal{Q}$.  The high-level idea behind our proof of Theorem~\ref{TheoremDLCoalIdent} is the following: 
\begin{quote}
	\textit{Conditioning on the number of copies in species $A,B,C,D$ in the species in $\mathcal{Q}$, independently pick a uniformly random gene copy $a,b,c,d$ in species $A,B,C,D$ and let $q$ be the corresponding quartet topology under $t^{\mathcal{Q}}$. We show that the most likely outcome is $q = ab|cd$.} 
\end{quote} 
This is the same approach as that used in~\cite{Legried821439}, but the analysis of the model is more involved.

Define $\mathcal{X} = (\mathcal{A},\mathcal{B},\mathcal{C},\mathcal{D})$ to be the number of copies in species $A,B,C,D$, respectively.  We will let $\P'$ be the probability measure subject to this conditioning.  While we only reconstruct an unrooted species tree, both locus trees and gene trees are in fact generated from rooted species trees. On a restriction to four species, there are \textit{two cases} of root location to consider: when the root is located on the internal quartet edge (the balanced case) or on the pendant edge incident with $D$ (the caterpillar case).  For gene copies $a,b,c,d$ from $A, B, C, D$, define the following events:  
\begin{align*}
Q_1 = \{q = ab|cd\}, \quad
Q_2 = \{q = ac|bd\}, \quad
Q_3 &= \{q = ad|bc\}.
\end{align*}

\subsection{Balanced case}

Let $R$ be the most recent common ancestor of $\mathcal{Q}$ in the species quartet $T^{\mathcal{Q}}$ and $I$ be the number of locus copies exiting $R$ (forward in time).  Let $\P''$ be the probability measure indicating conditioning on $I$ as well as $\mathcal{A},\mathcal{B},\mathcal{C},\mathcal{D}$.  For any selection of copies $(a,b,c,d)$ from each species in the quartet, let $i_x \in \{1,...,I\}$ be the ancestral lineage of $x \in \{a,b,c,d\}$ in $R$.
By the law of total probability, we have 
\begin{align} 
\label{EqnTotalProb''}
    \P[Q_1] = \E[\P''[Q_1]].
\end{align}  
Hence, in order to show identifiability of the species quartet, it is sufficient to show that 
\begin{align*} 
    \P[Q_1] > \max\{\P[Q_2],\P[Q_3]\}
\end{align*} 
when the copies of $(A,B,C,D)$ are chosen uniformly at random.  If $\mathcal{X} < \vec{1}$ (that is, at least one of $(A,B,C,D)$ fails to have a copy to select), then ASTRAL-one selects $Q_1,Q_2,Q_3$ each with probability $0$. So we consider the case $\mathcal{X} \geq \vec{1}$. We will use the notation $z_1 \land z_2 = \min\{z_1, z_2\}$.
\begin{prop}[Quartet identifiability: Balanced case] \label{PropClaimBal} 
Let $x = \P''[i_a = i_b]$ and $y = \P''[i_c = i_d]$.
On the events $\mathcal{X} \geq \vec{1}$ and $I \geq 1$,
we have almost surely 
\begin{align*}
    \P''[Q_1] - \P''[Q_2] > \frac{1}{12}\left(x - \frac{1}{I}\right) \wedge \left(y - \frac{1}{I}\right).
\end{align*}
\end{prop}
The next lemma establishes that $x,y \geq 1/I$, similarly to \cite[Lemma 1]{Legried821439}. 
\begin{lemma} 
\label{LemmaLeg1Copies} 
Let $x = \P''[i_a = i_b]$ and $y = \P''[i_c = i_d]$.  On the events $\mathcal{X} \geq \vec{1}$ and $I \geq 1$,
we have almost surely 
\begin{align*}
x \land y \geq \frac{1}{I}.
\end{align*}
\end{lemma}
\begin{proof}
For $j \in \{1,...,i\}$, let $N_j$ be the number of gene copies descending from $j$ in $R$ that survive to the most recent common ancestor of $A$ and $B$.  Conditioning on $(N_j)_{j}$, the choice of $a$ and $b$ in the locus tree is independent as in \cite{Legried821439}.  So $i_a$ and $i_b$ are picked proportionally to the $N_{j}$'s by symmetry.  Then \begin{align*}
    \P''[i_a = i_b] = \E''[\P''[i_a=i_b|(N_{j})_{j}]] = \E''\left[ \frac{\sum_{j=1}^{I}N_j^2}{(\sum_{j=1}^{I}N_j)^2}\right] \geq \frac{1}{I},
\end{align*} 
as in \cite[Lemma 1]{Legried821439}.
The same holds for $y$, completing the proof of the lemma.
\end{proof}

\subsubsection{Ancestral locus configurations} 

Conditioned on $\mathcal{X}$ and $I$, we will characterize the occurrence of $Q_1,Q_2,Q_3$ based on how $i_{x}, x \in \{a,b,c,d\}$ are selected at the root $R$.  Then, in a worst-case scenario, we will analyze coalescent events of the coalescent process above $R$.  For an arbitrary quartet $(a,b,c,d)$, we relate the likelihood of $Q_1,Q_2,Q_3$ under each of the following events: \begin{align}
\nonumber E &= a - b - c - d \\
\nonumber F_{ab} &= ab - c - d & F_{ac} &= ac - b - d & F_{ad} &= ad - b - c \\
\label{LocusTopR} F_{bc} &= bc - a - d & F_{bd} &= bd - a - c & F_{cd} &= cd - a - b \\
\nonumber G_{ab} &= ab - cd & G_{ac} &= ac - bd & G_{ad} &= ad - bc \\
\nonumber H_{abc} &= abc - d & H_{abd} &= abd - c & H_{acd} &= acd - b & H_{bcd} &= bcd - a \\
\nonumber K & = abcd,
\end{align}
where $-$ indicates separate lineages at $R$ for the chosen copies from $A,B,C,D$.  For example, the event $E$ indicates that the $i_a,i_b,i_c,i_d$ are distinct.  These events are disjoint and mutually exhaustive.  Letting $\mathcal{E}$ run across all the above events, the law of total probability implies 
\begin{align}
\label{Eqn:TotalProbE} 
\P''[Q_i] = \sum_{\mathcal{E}}\P''[Q_i|\mathcal{E}]\P''[\mathcal{E}].
\end{align}

\subsubsection{Reduction to coalescence above $R$}

For the rooted locus quartet implied by the four copies $a,b,c,d$, let $\nocoal$ be the event that no coalescent event occurs beneath $R$ between the four corresponding lineages. The following lemma shows that conditioning on $\nocoal$ reduces the probability of $Q_1$ while increasing that of $Q_2$.
\begin{lemma} 
\label{LemmaCondEcoalBal}
	For any $I$ and any $\mathcal{X} \geq \vec{1}$ and any event  $$\mathcal{E}\in \{E, F_{ab},\ldots,G_{ab},\ldots, H_{abc},\ldots\},$$
	we have 
\begin{align*}
	\P''[Q_1|\mathcal{E}] \geq \P''[Q_1|\mathcal{E}\cap\nocoal]
	\quad \text{and}\quad 
	\P''[Q_i|\mathcal{E}] \leq \P''[Q_i|\mathcal{E}\cap\nocoal], \quad i \in \{2,3\},
\end{align*}
almost surely.
\end{lemma}
\begin{proof}
	For $Q_1$, the law of total probability implies \begin{align*}
	\P''[Q_1|\mathcal{E}] = \P''[\nocoal^c|\mathcal{E}] + \P''[Q_{1}|\mathcal{E}\cap\nocoal]
	\,\P''[\nocoal|\mathcal{E}] \geq \P''[Q_1|\mathcal{E}\cap\nocoal],
	\end{align*} 
	where we used that $Q_1$ is guaranteed under $\nocoal^c$.  
    Similarly
	\begin{align*}
	\P''[Q_i|\mathcal{E}] 
	= \P''[Q_i|\mathcal{E}\cap\nocoal]
	\,\P''[\nocoal|\mathcal{E}] \leq \P''[Q_i|\mathcal{E}\cap\nocoal],
	\end{align*} 
	for $i \in \{2,3\}$.
\end{proof} 

The event $K$ will play a special role in the proof and we treat it separately. For the other terms, combining \eqref{Eqn:TotalProbE} and Lemma \ref{LemmaCondEcoalBal}, we have 
\begin{align*}
    \P''[Q_1] - \P''[Q_2] 
    &= (\P''[Q_1|K] - \P''[Q_2|K]) \P''[K]\\
    & \quad + \sum_{\mathcal{E} \neq K}\left(\P''[Q_1|\mathcal{E}]-\P''[Q_2|\mathcal{E}]\right)\P''[\mathcal{E}] \\
    &\geq (\P''[Q_1|K] - \P''[Q_2|K]) \P''[K]\\
    & \quad + \sum_{\mathcal{E} \neq K}\left(\P''[Q_1|\mathcal{E}\cap\nocoal] - \P''[Q_2|\mathcal{E}\cap\nocoal]\right)\P''[\mathcal{E}].
\end{align*} 
To prove Proposition \ref{PropClaimBal}, we derive an explicit bound on this last sum.
	
Under $\P''$, the events $E,H$ are symmetric in the sense that switching the roles of $a$ and $c$ or the roles of $a$ and $d$ does not change the conditional probability of $Q_{1}$ and $Q_{2}$.  Hence 
\begin{align*}
    \P''[Q_1|\mathcal{E}\cap\nocoal]
    =
    \P''[Q_2|\mathcal{E}\cap\nocoal],
    \quad \forall \mathcal{E}\in\{E,H\}
\end{align*}
and using this above we get
\begin{align}
    \P''[Q_1] - \P''[Q_2] 
    &\geq (\P''[Q_1|K] - \P''[Q_2|K]) \P''[K]\nonumber \\
    & \quad + \sum_{j \in \{ab,ac,ad\}} 
    \left(\P''[Q_1|G_{j}\cap\nocoal] - \P''[Q_2|G_{j}\cap\nocoal]\right)\P''[G_j]\nonumber \\
    &\quad + \sum_{j \in \{ab,...,cd\}}
    \left(\P''[Q_1|F_j\cap\nocoal] - \P''[Q_2|F_j\cap\nocoal]\right)\P''[F_j]. \label{Eqn:TotalProbE-FG}
\end{align}

\subsubsection{The $F$ and $G$ events}

We now consider the events $\{F_{ab},F_{cd},F_{ac},F_{bd}\}$ and $\{G_{ab},G_{ac}\}$.  

\paragraph{Event probabilities}
In the next lemma, we compute the probabilities of a given locus tree quartet satisfying the events in $\{F_{ab},...,F_{cd},G_{ab},G_{ac}\}$.  

\begin{lemma} 
\label{LemmaAncBal}
    Let $x = \P''[i_a = i_b]$ and $y = \P''[i_c = i_d]$.
	For $I \geq 2$ and any $\mathcal{X}\geq\vec{1}$, the following hold almost surely:
	\begin{align*}
	\nonumber \P''[F_{ab}]  &= \frac{I-2}{I}x(1-y) \\
	\nonumber \P''[F_{cd}] &= \frac{I-2}{I}(1-x)y \\
	\nonumber \P''[F_{ac}]  = \P''[F_{bd}] &= \frac{I-2}{I(I-1)}(1-x)(1-y) \\
	\nonumber \P''[G_{ab}] &= \frac{I-1}{I}xy \\
	\P''[G_{ac}] &= \nonumber \frac{1}{I(I-1)}(1-x)(1-y)
	\end{align*}
\end{lemma}

\begin{proof}
	The calculations for $F_{ab}$ and $F_{cd}$ are similar, except that we condition on different events. Indeed, note that
	\begin{align}
	\nonumber \P''[F_{ab}] &= \P''[F_{ab}|i_a = i_b, i_c \ne i_d] \P''[i_a = i_b]\P''[i_c \ne i_d] \\
	\nonumber \P''[F_{cd}] &= \P''[F_{cd}[i_a \ne i_b, i_c = i_d] \P''[i_a \ne i_b]\P''[i_c = i_d].
	\end{align} The conditional probability of $F_{ab}$ is then obtained by considering that given the placement of the pair $(i_c,i_d)$ among the $I$ ancestral lineages, the shared lineage $i_a = i_b$ has $I-2$ choices where they do not intersect $\{i_c,i_d\}$.  The result in the statement follows.  
	Similarly, for $F_\iota$ with $\iota \in \{ac,bd\}$, we have
	\begin{align}
	\nonumber \P''[F_\iota] &= \P''[F_\iota|i_a \ne i_b, i_c \ne i_d]\P''[i_a \ne i_b]\P''[i_c \ne i_d].
	\end{align} In this case, out of $I(I-1)$ choices for $i_a$ and $i_b$, the choice of $i_c$ is determined and there are $I-2$ remaining choices for $i_d$, implying the result.
	
	We use the same principle for $G_{ab}$ and $G_{ac}$.  Keeping this in mind, we have \begin{align}
	\nonumber \P''[G_{ab}] &= \P''[G_{ab}|i_a = i_b,i_c = i_d]\P''[i_a=i_b]\P'_{I}[i_c = i_d] \\
	\nonumber \P''[G_{ac}] &= \P''[G_{ac}|i_a \ne i_b, \nonumber i_c \ne i_d]\P''[i_a \ne i_b]\P''[i_c \ne i_d],
	\end{align} and we proceed as before to get the result.
\end{proof}

Using the previous lemma, we collect further bounds on the probabilities of events at the root of the locus tree.
\begin{lemma} 
\label{LemmaDifferencesBal} 
Letting again $x = \P''[i_a = i_b]$ and $y = \P''[i_c = i_d]$, the following statements hold.
\begin{itemize}
\item[(a)] If $I=2$ then \begin{equation*}\P''[G_{ab}] - \P''[G_{ac}] \geq  \left(x-\frac{1}{2}\right)\wedge\left(y-\frac{1}{2}\right)\end{equation*}
\item[(b)] If $I\geq 3$ and $x \land y \geq 1/2$, then \begin{equation*}\P''[G_{ab}] - \P''[G_{ac}] \geq 1/8.\end{equation*}
\item[(c)] If $I = 2$, \begin{equation*}\P''[F_{ab}] - \P''[F_{ac}] - \P''[F_{bd}] + \P''[F_{cd}] = 0.\end{equation*}
\item[(d)] If $I\geq 3$ and $x \land y \leq 1/2$, then \begin{equation*}\P''[F_{ab}] - \P''[F_{ac}] - \P''[F_{bd}] + \P''[F_{cd}] \geq \frac{1}{4}\left(x-\frac{1}{I}\right)\wedge\left(y-\frac{1}{I}\right).\end{equation*} 
\end{itemize}
\end{lemma}

\begin{proof} 
By Lemma \ref{LemmaAncBal}, $\P''[G_{ab}] - \P''[G_{ac}] =  \frac{1}{2}(x+y-1)$ which implies part (a). 

To prove (b), observe that by Lemma \ref{LemmaAncBal} again,
\begin{align*} 
  \P''[G_{ab}] - \P''[G_{ac}] 
  &= \frac{I-1}{I}xy-\frac{1}{I(I-1)}(1-x)(1-y)\\
  &\geq \frac{1}{2}\left(\frac{I-1}{I}y - \frac{1}{I(I-1)}(1-y) \right)\\
  &\geq \frac{1}{4}\left(\frac{I-1}{I}-\frac{1}{I(I-1)}\right)\\
  &=\frac{1}{4}\left(\frac{I-2}{I-1}\right)
\end{align*}
where the inequalities are justified by the assumption $x \land y \geq 1/2$. Since $I\geq 3$, it follows that $\P''[G_{ab}] - \P''[G_{ac}] \geq 1/8$.

To prove (c) and (d), observe that by Lemma \ref{LemmaAncBal}, 
\begin{align}
    \nonumber &\P''[F_{ab}] - \P''[F_{ac}] - \P''[F_{bd}] + \P''[F_{cd}]\\
    \nonumber &= \frac{I-2}{I}\left(x(1-y) + y(1-x)\right) - 2 \frac{I-2}{I(I-1)}(1-x)(1-y) \\
    \nonumber &= \frac{I-2}{I(I-1)}\left((1-y)\left((I-1)x - (1-x)\right) + (1-x)\left((I-1)y - (1-y)\right)\right) \\
    \nonumber &= \frac{I-2}{I(I-1)}\left((1-y)(Ix - 1) + (1-x)(Iy - 1)\right).
\end{align} 
Clearly if $I=2$, the right-hand side is zero, which proves (c). Furthermore, since $x,y\geq 1/I$ by Lemma~\ref{LemmaLeg1Copies}, 
it follows that both $(1-y)(Ix - 1)\geq 0$ 
and $(1-x)(Iy - 1)\geq 0$, and therefore
\begin{equation*}
  \P''[F_{ab}] - \P''[F_{ac}] - \P''[F_{bd}] + \P''[F_{cd}] \geq \frac{I-2}{I(I-1)}\left(1-u\right)\left(Iv  - 1\right)
\end{equation*}
for $(u,v)\in \{(x,y), (y,x)\}$. Taking $u=\min(x,y)$ and $v=\max(x,y)$ gives
\begin{align*}
  \P''[F_{ab}] - \P''[F_{ac}] - \P''[F_{bd}] + \P''[F_{cd}] 
  &\geq \frac{I-2}{I(I-1)}\left(1-\min(x,y)\right)\left(I\max(x,y)  - 1\right) \\
  & \geq \frac{I-2}{I-1} \left(1-\min(x,y)\right)\left(\max(x,y) - \frac{1}{I}\right) \\
  & \geq \frac{1}{2}\left(1-\min(x,y)\right)\left(\max(x,y) - \frac{1}{I}\right) \\
  & \geq \frac{1}{4}\left(\max(x,y) - \frac{1}{I}\right)
\end{align*}	
which implies (d).
\end{proof}

\paragraph{Conditional probabilities of quartet topologies} In the following lemma, we give expressions for $\P''[Q_i|\mathcal{E}\cap\nocoal]$ across the events $\{F_{ab},F_{cd},F_{ac},F_{bd}\}$ and $\{G_{ab},G_{ac}\}$.
\begin{lemma} 
\label{LemmaSymmetryBal}
	(a) For any $I$ and any $\mathcal{X} \geq \vec{1}$, we have \begin{align*}
	\P''[Q_1 | F_{ab}\cap\nocoal]
	= \P''[Q_1|F_{cd}\cap\nocoal]
	= \P''[Q_2 | F_{ac}\cap\nocoal]
	= \P''[Q_2 | F_{bd}\cap\nocoal]
	:= \phi''_+
	\end{align*} 
	and
	\begin{align*}
	\P''[Q_2 | F_{ab}\cap\nocoal]
	= \P''[Q_2| F_{cd}\cap\nocoal]
	= \P''[Q_1 | F_{ac}\cap\nocoal]
	= \P''[Q_1 | F_{bd}\cap\nocoal]
	:= \phi''_-.
	\end{align*}

	(b) For any $I$ and any $\mathcal{X} \geq \vec{1}$, we have \begin{align*}
	\P''[Q_1|G_{ab}\cap\nocoal] = 1
	\end{align*} 
	and 
	\begin{align*}
	\P''[Q_2|G_{ac}\cap\nocoal] = 1.
	\end{align*}
\end{lemma}

\begin{proof}
	(a) The quantities $\phi''_{+}$ and $\phi''_-$ are indeed well-defined as above by symmetry.  (b) By switching the roles of $b$ and $c$, we observe that $\P''[Q_{1}|G_{ab}\cap\nocoal] = \P''[Q_{2}|G_{ac}\cap\nocoal]$.  To see why $\P''[Q_{1}|G_{ab}\cap\nocoal] = 1$, we again examine the topology above the root with leaves $ab$ and $cd$.  At least one of these leaves descends from a daughter edge, which implies $Q_{1}$ is constructed with probability $1$.  This completes the proof of the lemma.
\end{proof}

The following lemma establishes that, conditioned on $F_{ab}$ and $\nocoal$, the difference in probability between $Q_1$ and $Q_2$ is at least $1/3$. 

\begin{lemma}\label{phi-difference-lemma}
For $I \geq 1$ and any $\mathcal{X}\geq \vec{1}$, we have \begin{align*}
\phi''_{+} - \phi''_{-} \geq \frac{1}{3}.
\end{align*} 
\end{lemma}
\begin{proof}
By definition of $\phi''_{+}$ and $\phi''_{-}$, it suffices to show $\P''[Q_1|F_{ab}\cap\nocoal] - \P''[Q_2|F_{ab}\cap\nocoal]\geq 1/3$. Conditioned on $F_{ab}\cap\nocoal$, no coalescent event between the chosen lineages occurs beneath $R$.  So, we examine the topology of the locus tree above the root with leaf set being the three leaves implied by $F_{ab}$.  Using the law of total probability, we condition further across the three possible rooted locus topologies on the three leaves $ab$, $c$, and $d$.  Let $\tau_{i}$ be the rooted topology in which character $i$ is the outgroup of the triple.  Using Newick tree format, for example we have $\tau_{ab} = ((c,d),ab)$.  Then 
\begin{align*}
&\P''[Q_1|F_{ab}\cap\nocoal] 
- \P''[Q_2|F_{ab}\cap\nocoal] \\
& \quad = 
\frac{1}{3}\sum_{i}\left( 
\P''[Q_{1}|\tau_{i},F_{ab}\cap\nocoal] - \P''[Q_{2}|\tau_{i},F_{ab}\cap\nocoal]\right),
\end{align*} 
where we used the fact that 
$\P''[\tau_{i}|F_{ab}\cap\nocoal] = 1/3$ for each $i$. 
Now we compute the summands.  If $i = ab$, then either $ab$ descends from a daughter lineage or the pair $(c,d)$ descends from a daughter lineage, meaning we observe $Q_1$ with probability $1$ and $Q_2$ with probability $0$.  In the other two cases, let $p>0$ be the probability that $a$ and $b$ coalesce along the pendant edge for $ab$.  If they do not coalesce along the pendant edge, then the lineages from $a$ and $b$ live in the same population as that of $c$.  Then there is probability $1/3$ that the first coalescing pair among $a,b,c$ is $a,b$.  So the probability of observing $Q_1$ is $p+\frac{1}{3}(1-p)$.  There is probability $1/3$ that the first coalescing pair among $a,b,c$ is $a,c$, so the probability of observing $Q_2$ is $\frac{1}{3}(1-p)$. Then \begin{align*}
    \phi_{+}^{''} - \phi_{-}^{''} = \frac{1}{3}\left(1 - 0 + 2\left(p +\frac{1}{3}(1-p) - \frac{1}{3}(1-p)\right)\right) = \frac{1}{3}(1+2p) \geq \frac{1}{3}.
\end{align*}
\end{proof}

\subsubsection{Proof of Proposition \ref{PropClaimBal}}

With that we can prove Proposition \ref{PropClaimBal}.
\begin{proof}[Proof of Proposition \ref{PropClaimBal}]
In the $I = 1$ case, 
$\P''[K] = 1$
so 
\begin{align*}
    \P''[Q_1] - \P''[Q_2] 
    =
    \P''[Q_1|K] 
    - 
    \P''[Q_2|K]
    > 0,
\end{align*} 
where we used that, under $K\cap\nocoal$, the quartets $Q_1$ and $Q_2$ occur with equal probability under $\P''$. Since $x - 1/I = y - 1/I = 0$, the claim follows. 

For $I \geq 2$, \eqref{Eqn:TotalProbE-FG} and Lemma \ref{LemmaSymmetryBal} implies that
\begin{align*} 
\P''[Q_1] - \P''[Q_2] 
&\geq (\P''[Q_1|K] - \P''[Q_2|K]) \P''[K] \\
& \quad + \P''[G_{ab}] - \P''[G_{ac}]  + \left(\phi''_{+} - \phi''_{-} \right)\left(\P''[F_{ab}]+\P''[F_{cd}]\right)\\
& \quad - \left(\phi''_{+} - \phi''_{-} \right)\left(\P''[F_{ac}]+\P''[F_{bd}]\right)\\
& > \P''[G_{ab}] - \P''[G_{ac}] \\
&\quad + \left(\phi''_{+} - \phi''_{-} \right)\left(\P''[F_{ab}]-\P''[F_{ac}]-\P''[F_{bd}] + \P''[F_{cd}]\right),
\end{align*}
where again we used that, under $K\cap\nocoal$, the quartets $Q_1$ and $Q_2$ occur with equal probability.
If $I = 2$, then by Lemma \ref{phi-difference-lemma} and Lemma \ref{LemmaDifferencesBal} parts (a) and (c), this leads to 
\begin{align*}
    \P''[Q_1] - \P''[Q_2] &> \left(x - \frac{1}{I}\right) \wedge \left(y - \frac{1}{I}\right).
\end{align*}  
If $I \geq 3$, then by Lemma \ref{LemmaDifferencesBal} parts (b) and (d),
$$
\P''[Q_1] - \P''[Q_2] 
> 
\begin{cases} 1/8 &\mbox{if } x \land y\geq 1/2 \\
\frac{1}{12}\left(x-\frac{1}{I}\right)\wedge\left(y-\frac{1}{I}\right) &\mbox{if } x \land y\leq 1/2 \end{cases} $$
It follows that $\P''[Q_1] - \P''[Q_2]> \frac{1}{12}\left(x-\frac{1}{I}\right)\wedge\left(y-\frac{1}{I}\right)$, finishing the proof of the main claim in the balanced case.  
\end{proof}

\subsection{Caterpillar case}

We now consider the caterpillar case.
Without loss of generality, assume the species tree restricted to $\mathcal{Q}$ has topology $(((A,B),C),D)$.  Let $R$ be the most recent common ancestor of $A,B,C$ and let $I$ be the number of locus copies exiting $R$ (forward in time).  Let $\P''$ be the probability measure indicating conditioning on $I$ and $\mathcal{X}$.
Let $i_{x} \in \{1,...,I\}$ be the ancestral lineage of $x \in \{a,b,c\}$ in $R$.    
As with the balanced case, if $\mathcal{X} < \vec{1}$, then ASTRAL-one selects $Q_1,Q_2,Q_3$ each with probability $0$.  To prove \begin{align*}
    \P[Q_1] > \max\{\P[Q_2],\P[Q_3]\},
\end{align*} it is sufficent to prove Proposition \ref{PropClaimCat} below for $\mathcal{X} \geq \vec{1}$.
\begin{prop}[Quartet identifiability: Caterpillar case] 
\label{PropClaimCat}
Let $x = \P''[i_a = i_b]$.
On the events $I \geq 1$ and $\mathcal{X} \geq \vec{1}$, we have almost surely
\begin{align*}
    \P''[Q_1] - \P''[Q_2] > \frac{1}{3}\left(x - \frac{1}{I}\right).
\end{align*}
\end{prop}

Similarly to the balanced case, in order to prove this proposition we consider the following events: 
\begin{align}
    \nonumber E &= a - b - c \\
    \nonumber G_{ab} &= ab - c & G_{ac} &= ac - b & G_{bc} &= bc - a \\
    \nonumber K &= abc,
\end{align} 
where $-$ indicates separation of lineages in $R$ of the chosen copies of $A,B,C$.  Letting $\mathcal{E}$ run across all events, the law of total probability implies \begin{align} \label{eq:TotalProbECat}
    \P''[Q_i] = \sum_{\mathcal{E}}\P''[Q_i|\mathcal{E}]\P''[\mathcal{E}].
\end{align}
Let $\nocoal$ be the event that no coalescent event occurs beneath $R$ between the three lineage corresponding to $a, b, c$.   

Analogues to Lemmas
\ref{LemmaLeg1Copies},
\ref{LemmaCondEcoalBal},
\ref{LemmaAncBal},
\ref{LemmaDifferencesBal},
and 
\ref{LemmaSymmetryBal} 
hold with similar proofs.
\begin{lemma} 
\label{LemmaLeg1CopiesCat} 
Let $x = \P''[i_a = i_b]$.  On the events $\mathcal{X} \geq \vec{1}$ and $I \geq 1$,
we have almost surely 
\begin{align*}
x \geq \frac{1}{I}.
\end{align*}
\end{lemma}
\begin{lemma} \label{LemmaCondEcoalCat}
  Let the species tree be a rooted caterpillar on four leaves $A,B,C,D$.  For all $I \geq 1$ and $\mathcal{X} \geq \vec{1}$, and any event $\mathcal{E}\in\{E, G_{ab},G_{ac},G_{bc}\}$,
  \begin{align*}
      &\P''[Q_1|\mathcal{E}] \geq \P''[Q_1|\mathcal{E}\cap\nocoal]\quad\text{and}\quad \P''[Q_i|\mathcal{E}] \leq \P''[Q_i|\mathcal{E}\cap\nocoal], \quad i \in \{2,3\}.
  \end{align*}
\end{lemma}
\begin{lemma} \label{LemDecompONECat}
For $I \geq 2$ and any $\mathcal{X} \geq \vec{1}$, let $x = \P''[i_a = i_b]$.  Then the following hold:
\begin{align}
    \nonumber \P''[G_{ab}] &= \frac{I-1}{I}x \\
    \nonumber \P''[G_{ac}] = \P''[G_{bc}] &= \frac{1}{I}(1-x).
\end{align}
\end{lemma}
\begin{lemma}\label{second-caterpillar-inequality}
  On $I\geq 1$, almost surely
  \begin{equation*}
    \P''[G_{ab}] - \P''[G_{ac}] = x-\frac{1}{I}.
  \end{equation*}
  \end{lemma}
\begin{lemma}
 \label{LemmaSymmetryCat}
 For any $I$ and any $\mathcal{X} \geq \vec{1}$, we have \begin{align*}
     \P''[Q_{1}|G_{ab}\cap\nocoal] = \P''[Q_2|G_{ac}\cap\nocoal] := \psi''_{+}
 \end{align*} and \begin{align*}
     \P''[Q_{2}|G_{ab}\cap\nocoal] = \P''[Q_1|G_{ac}\cap\nocoal] := \psi''_{-}.
 \end{align*}
\end{lemma}

\subsubsection{The $G$ events}

The following lemma bounds the conditional probability difference for the $G$ events.
\begin{lemma}\label{first-caterpillar-inequality}
  On the events $I \geq 1$ and $\mathcal{X} \geq \vec{1}$, we have almost surely
\begin{align*}
  \psi''_{+}- \psi''_{-}\geq \frac{1}{3}.
\end{align*}
\end{lemma}
\begin{proof}
By definition of $\psi''_{+}$ and $\psi''_{-}$, it suffices to show that $\P''[Q_1|G_{ab}\cap\nocoal]-\P''[Q_2|G_{ab}\cap\nocoal] \geq 1/3$. 
\begin{figure}
    \centering
    \includegraphics[scale=.7]{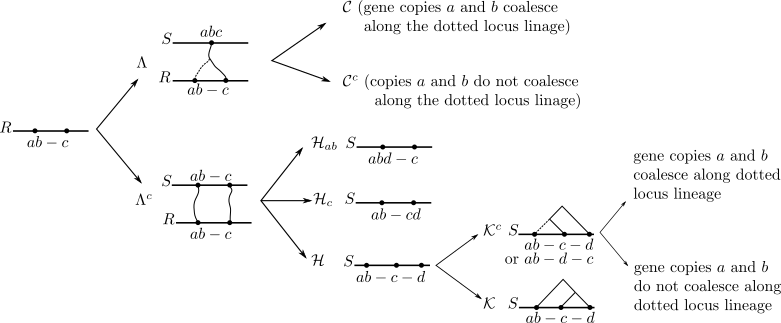}
    \caption{Flowchart for case analysis in Lemma~\ref{first-caterpillar-inequality}.}
    \label{fig:cat-diff}
\end{figure}
The proof of this inequality involves decomposing $G_{ab}\cap\nocoal$ into a number of subcases, depicted in Figure~\ref{fig:cat-diff}, and computing the probabilities of $Q_1$ and $Q_2$ in each subcase. Let $S$ be the most recent common ancestor of $\mathcal{Q}$ in the species tree. Let $\Lambda$ be the event that the $ab$ and $c$ individuals in $R$ descend from a common ancestor in $S$ and let $q=\P''[\Lambda|G_{ab}\cap\nocoal]$. There are two cases:
\begin{enumerate}
\item \textit{(Condition on $\Lambda$)}
Let $\mathcal{C}$ be the event that gene copies $a$ and $b$ coalesce above $R$ and below the MRCA of loci $i_{ab}$ and $i_{c}$.  Let $q' = \P''[\mathcal{C}|\Lambda,G_{ab},\nocoal]$.
We claim that $q'\geq 1/2$. To see this, observe that conditional on $G_{ab}\cap\nocoal$ the loci $i_{ab}$ and $i_{c}$ share the same ancestral locus at $S$ only if there occurred a duplication event between $S$ and $R$ which is ancestral to both of them. Therefore with probability at least $1/2$, the gene copies $a,b$ coalesce along their shared pendant edge in the rooted topology between $R$ and $S$, proving the claim. Furthermore, it is obvious that
\begin{equation}
\P''[Q_1|\mathcal{C},\Lambda,G_{ab},\nocoal]-\P''[Q_2|\mathcal{C},\Lambda,G_{ab},\nocoal] = 1.\label{eq:diff-FE}
\end{equation}
On the other hand, conditional on $\mathcal{C}^c$, the copies of $a,b$ and $c$ enter the same population and are then symmetric, and hence
\begin{equation}
\P''[Q_1|\mathcal{C}^c,\Lambda,G_{ab},\nocoal]-\P''[Q_2|\mathcal{C}^c,\Lambda,G_{ab},\nocoal] = 0.\label{eq:diff-F'E}
\end{equation}

\item \textit{(Condition on $\Lambda^c$)} 
Let $\mathcal{H}_{j}$ be the event that copies $d$ and $j$ share the same ancestor in the locus tree at $S$, and define $\mathcal{H}=(\mathcal{H}_{ab}\cup \mathcal{H}_{c})^c$ and  $r= \P''[\mathcal{H}|\Lambda^c,G_{ab},\nocoal]$. Then by symmetry, $\P''[\mathcal{H}_{ab}|\Lambda^c,G_{ab},\nocoal]= \P''[\mathcal{H}_{c}|\Lambda^c,G_{ab},\nocoal]=\frac{1-r}{2}$.
By a further symmetry argument similar to that made in Case 1 we have
\begin{equation}
\P''[Q_1|\mathcal{H}_{ab},\Lambda^c,G_{ab},\nocoal]- \P''[Q_2|\mathcal{H}_{ab},\Lambda^c,G_{ab},\nocoal]\geq 0,\label{eq:diff-HabE'}
\end{equation}
where the inequality accounts for the possibility that the lineages from $a$ and $b$ coalesce between $R$ and $S$.
Let $\tau$ be the topology of the locus tree restricted to the copies $ab,c,d$ and restricted to the portion above $S$ (and suppressing nodes of degree 2). Conditioned on $\mathcal{H}_{c}$, we have $\tau = (ab,cd)$, so it must be the case that either $ab$ descends from a daughter lineage or $cd$ descends from a daughter lineage, meaning we observe $Q_1$ with probability $1$ and $Q_2$ with probability $0$. Therefore
\begin{equation}
\P''[Q_1|\mathcal{H}_{c},\Lambda^c,G_{ab},\nocoal]- \P''[Q_2|\mathcal{H}_{c},\Lambda^c,G_{ab},\nocoal]= 1.\label{eq:diff-HcE'}
\end{equation}
It remains to consider the case $\mathcal{H} = (\mathcal{H}_{ab}\cup \mathcal{H}_{c})^c$. Let $\mathcal{K}$ be the event that $\tau=(ab,(c,d))$. By symmetry, the three possible topologies are equally likely, so $\P''[\mathcal{K}|\mathcal{H},\Lambda^c,G_{ab},\nocoal] = 1/3$. Conditioned on $\mathcal{K}$, either $ab$ descends from a daughter lineage or the pair $(c,d)$ descends from a daughter lineage, and therefore
\begin{equation}
\P''[Q_1|\mathcal{K},\mathcal{H},\Lambda^c,G_{ab},\nocoal]- \P''[Q_2|\mathcal{K},\mathcal{H},\Lambda^c,G_{ab},\nocoal]= 1.\label{eq:diff-KHE'}
\end{equation}
Conditioned on $\mathcal{K}^c$, let $p$ denote the probability that gene copies $a$ and $b$ coalesce along the pendant edge for $ab$ in the rooted triple above $S$. Then $$\P''[Q_1|\mathcal{K}^c,\mathcal{H},\Lambda^c,G_{ab},\nocoal]= p+\frac{1}{3}(1-p),$$ 
and 
$$\P''[Q_2|\mathcal{K}^c,\mathcal{H},\Lambda^c,G_{ab},\nocoal] = \frac{1}{3}(1-p),$$ 
and hence
\begin{equation}
\P''[Q_1|\mathcal{K}^c,\mathcal{H},\Lambda^c,G_{ab},\nocoal]- \P''[Q_2|\mathcal{K}^c,\mathcal{H},\Lambda^c,G_{ab},\nocoal]\geq p,\label{eq:diff-K'HE'}
\end{equation}
where again the inequality accounts for the possibility that the lineages from $a$ and $b$ coalesce between $R$ and $S$.

\end{enumerate}

Finally, applying the law of total probability and using equations (\ref{eq:diff-FE})-(\ref{eq:diff-K'HE'}) gives
\begin{align*}
  \P''[Q_1|G_{ab},\nocoal]-\P''[Q_2|G_{ab},\nocoal]
  &\geq q'q+ \left(\frac{1-r}{2} + \frac{1}{3}r + \frac{2}{3}pr\right)(1-q) \\
  &\geq \frac{1}{2}q + \frac{1}{3}(1-q) = \frac{1}{3},
\end{align*}
where the inequality follows from $q'\geq 1/2$ and $r\geq0$. 
\end{proof}

\subsubsection{Proofs of Proposition \ref{PropClaimCat}
and Theorems \ref{TheoremDLCoalIdent} and \ref{TheoremDLCoalONE}}

With that, the proof of Proposition \ref{PropClaimCat}
is similar to that of Proposition \ref{PropClaimBal}.

In both the balanced and caterpillar cases, observe that $\P''[Q_1] -\P''[Q_3] = \P''[Q_1] - \P''[Q_2]$ by switching the roles of $c$ and $d$.  By Propositions \ref{PropClaimBal} and \ref{PropClaimCat}, all species quartet topologies are identifiable and hence
we have verified Theorem \ref{TheoremDLCoalIdent}.

Theorem \ref{TheoremDLCoalONE} then follows, along similar lines as \cite[Theorem 2]{Legried821439}, from the law of large numbers.

\subsection{Proof of consistency for ASTRAL-multi}

Before finishing the proof of Theorem~\ref{TheoremSampleONE}, we give a proof of Theorem~\ref{TheoremDLCoalMulti}.

\begin{proof}[Proof of Theorem~\ref{TheoremDLCoalMulti}]
Let $\mathcal{N}_{AB|CD}$ (respectively $\mathcal{N}_{AC|BD},\mathcal{N}_{AD|BC}$) be the number of choices consisting of one gene copy in the gene tree from each species in $\mathcal{Q} = \{A,B,C,D\}$ whose corresponding restriction in $t_1$ agrees with $AB|CD$ (respectively $AC|BD$, $AD|BC$).  Similarly to \cite[Theorem 2]{Legried821439}, it suffices to show that 
\begin{align}
\label{eq:multi-expec}
\E[\mathcal{N}_{AB|CD}] > \max\left\{\E[\mathcal{N}_{AC|BD}],\E[\mathcal{N}_{AD|BC}]\right\}.
\end{align} 

Letting again $\mathcal{X} = (\mathcal{A}, \mathcal{B}, \mathcal{C}, \mathcal{D})$, by taking expectation with respect to $I$ in Propositions \ref{PropClaimBal} and \ref{PropClaimCat}, we have on the event $\mathcal{X} \geq \vec{1}$ that
\begin{align}
\label{eq:multi-props}
     \P[q = AB|CD\,|\,\mathcal{X}] > \max\{\P[q = AC|BD\,|\,\mathcal{X}],
     \P[q = AD|BC\,|\,\mathcal{X}]\},
\end{align}
where $q$ is the topology of a uniformly chosen quartet among $A, B, C, D$. Let $\mathcal{M} = \mathcal{A}\mathcal{B}\mathcal{C}\mathcal{D}$ be the number of quartet choices and let $q_i, i=1,\ldots,\mathcal{M}$ be the corresponding topologies ordered arbitrarily. Because $q$ is a uniform choice, we have
\begin{align}
\label{eq:multi-unif}
    \P[q = AB|CD\,|\,\mathcal{X}]
    = \frac{1}{\mathcal{M}} \sum_{i=1}^{\mathcal{M}} \P[q_i = AB|CD\,|\,\mathcal{X}],
\end{align}
and similarly for the other topologies. Since 
\begin{align*}
    \mathcal{N}_{AB|CD}
    = \sum_{i=1}^{\mathcal{M}} \mathbf{1}\{q_i = AB|CD\},
\end{align*}
and similarly for the other topologies, taking expectations and using~\eqref{eq:multi-props} and~\eqref{eq:multi-unif} gives
\eqref{eq:multi-expec} as claimed.
\end{proof}

\section{Proof of sample complexity bound}
\label{section:sample}

To prove Theorem~\ref{TheoremSampleONE}, our sample complexity result for ASTRAL-one, we use a union bound over all quartets and build on the analysis of Section~\ref{section:first-step}.  In particular, the key step of the proof is a more careful analysis of the events $K$ that appeared in the proof of Theorem~\ref{TheoremDLCoalIdent}. We first discuss a number of quantities that play an important role in the analysis.

\subsection{Bounds on branching process quantities}

We highlight the role of a number of parameters in the sample complexity: the shortest branch length in the species tree, $f$; the depth of the species tree, $\Delta$; and the duplication and loss rates, $\lambda$ and $\mu$.  These parameters enter the analysis through three quantities of significance: 
\begin{itemize}
    \item \textit{Coalescence of a pair of lineages on an edge:} In the standard coalescent, the probability that a pair of lineages has coalesced by time $f$ is
    \begin{align*}
        \gamma = 1 - e^{-f}.
    \end{align*}

    \item \textit{Survival probability of a quartet:} For a quartet $\mathcal{Q} = \{A, B, C,D\}$, let $\mathcal{X}_\mathcal{Q} = (\mathcal{A}, \mathcal{B}, \mathcal{C}, \mathcal{D})$ be the number of gene copies in the corresponding species.  
    The smallest probability over all quartets that a gene family contains a copy in each species will be denoted by
\begin{align*}
    \sigma
    := \min_{\mathcal{Q}}
    \P\left[\mathcal{X}_{\mathcal{Q}} \geq \vec{1}\right].
\end{align*}

    \item \textit{Expected number of lineages at a vertex:} For any vertex $R$ in the species tree, let $I_R$ be the number of copies at $R$ in a single gene family. The largest expectation of $I_R$ over all vertices will be denoted by
    \begin{align*}
        \alpha = 
        \max_{R} \E[I_R].
    \end{align*}
\end{itemize}

These last two quantities can be controlled using branching process theory. See e.g.~\cite[Section 9.2]{steelbook2016} for the relevant results in the phylogenetic context. We use the notation $z_1 \lor z_2 = \max\{z_1, z_2\}$.
\begin{lemma}
  \label{lemma:sigma-bound} The following hold:
  \begin{itemize}
    \item When $\mu > \lambda$,
    \begin{align*}
      \sigma \geq  \left[ \frac{1}{e^{(\mu - \lambda)\Delta}} \left(1 -  \frac{\lambda}{\mu} \right) \right]^4.
  \end{align*}    
  \item When $\lambda > \mu$,
    \begin{align*}
      \sigma \geq  \left[ 1 - \frac{\mu}{\lambda}\right]^4 .
  \end{align*}    
  \end{itemize}
  
\end{lemma}
\begin{proof}
Let $\mathcal{Q} = \{A, B, C,D\}$ be a quartet and let as before $\mathcal{X}_\mathcal{Q} = (\mathcal{A}, \mathcal{B}, \mathcal{C}, \mathcal{D})$ be the number of gene copies in the corresponding species. Assume the species tree topology on $\mathcal{Q}$ is balanced (the argument in the caterpillar case being similar). The probability that $\mathcal{A} \geq 1$ is given by
\begin{align*}
    \P[\mathcal{A} \geq 1]
    = 1 - \frac{\mu}{\lambda} q(\Delta),
\end{align*}
where
\begin{align*}
    q(t) 
    =
    \lambda
    \frac{1 
    - 
    e^{-(\lambda - \mu) t}}
    {\lambda 
    - 
    \mu e^{-(\lambda - \mu) t}}.
\end{align*}
It can be checked that, whether $\mu > \lambda$ or $\mu < \lambda$, the function $q(t)$ is increasing in $t$. Conditioned on $\{\mathcal{A} \geq 1\}$, there is at least one copy in each vertex along the path between the root and $A$. Hence
\begin{align*}
    \P[\mathcal{D} \geq 1\,|\, \mathcal{A} \geq 1] \geq 1 - \frac{\mu}{\lambda} q(\Delta). 
\end{align*}
Repeating this argument for $B$ and $C$ gives
\begin{align*}
    \P\left[\mathcal{X}_{\mathcal{Q}} \geq \vec{1}\right]
    \geq \left(1 - \frac{\mu}{\lambda} q(\Delta)\right)^4.
\end{align*}
It remains to bound the right-hand side. When $\lambda > \mu$, $q(t) \to 1$ as $t \to +\infty$, which implies $q(t) \leq 1$ by monotonicity. So $1 - \frac{\mu}{\lambda} q(\Delta) \geq 1 - \frac{\mu}{\lambda}$. On the other hand, when $\mu > \lambda$, 
\begin{align*}
    1 - \frac{\mu}{\lambda} q(\Delta)
    &=
    \lambda
    \frac{1 
    - 
    e^{-(\lambda - \mu) \Delta}}
    {\lambda 
    - 
    \mu e^{-(\lambda - \mu) \Delta}}\\
    &= \frac{\mu - \lambda}{\mu e^{(\mu - \lambda)\Delta} - \lambda}\\
    &\geq \frac{1}{e^{(\mu - \lambda)\Delta}} \left(1 -  \frac{\lambda}{\mu} \right).
\end{align*}
\end{proof}

\begin{lemma}
  \label{lemma:alpha-bound} We have
  \begin{align*}
      \alpha \leq 1 \lor e^{(\lambda - \mu) \Delta}.
  \end{align*}
\end{lemma}
\begin{proof}
Recall that we assume there is a single lineage at the top pendant vertex of the species tree. If the time elapsed between this vertex and another vertex $U$ is $d$, then the expectation number of lineages at $U$ is $e^{(\lambda - \mu)d}$.   The result follows from the fact that $d \leq \Delta$ by considering separately the cases $\mu > \lambda$ and $\lambda > \mu$.
\end{proof}

\subsection{Sufficient effective number of samples}

For a quartet $\mathcal{Q}$,
let $\mathcal{K}_{\mathcal{Q}}$
be the set of gene trees such that
each species in $\mathcal{Q}$ has at least one gene copy. 
For $k^* \geq 1$, let
\begin{align*}
\mathcal{S}_{k^*}
    = 
    \left\{
    |\mathcal{K}_{\mathcal{Q}}| \geq k^*
    : \forall \mathcal{Q}
    \right\}.
\end{align*}

\begin{lemma}
\label{lemma:kstar}
For any $k^* \geq 1$ and $\epsilon \in (0,1)$, it holds that $\P[\mathcal{S}_{k^*}] \geq 1-\epsilon$ provided
\begin{align*}
    k \geq \left\{\frac{2 k^*}{\sigma}\right\}
    \lor
    \left\{\frac{8}{\sigma^2} \log \frac{n}{\epsilon}\right\}.
\end{align*}
\end{lemma}
\begin{proof}
For any $\mathcal{Q}$, by the definition of $\mathcal{K}_{\mathcal{Q}}$,
if $k$ is the number of loci then
\begin{align*}
    \E|\mathcal{K}_{\mathcal{Q}}| \geq k \sigma.
\end{align*}
Assume assume $k$ is large enough that $\frac{1}{2} k\sigma \geq k^*$. Then by Hoeffding's inequality (see, e.g. \cite{Vershynin:2018}),
\begin{align*}
    \P\left[
    \mathcal{S}_{k^*}^c
    \right]
    &\leq \sum_{\mathcal{Q}}
    \P\left[|\mathcal{K}_{\mathcal{Q}}| < k^* \right]\\    &\leq \sum_{\mathcal{Q}}
    \P\left[|\mathcal{K}_{\mathcal{Q}}| < \frac{1}{2} k \sigma \right]\\
    &\leq \sum_{\mathcal{Q}}
    \P\left[\E|\mathcal{K}_{\mathcal{Q}}| - |\mathcal{K}_{\mathcal{Q}}| > \frac{1}{2} k\sigma\right]\\
    &\leq n^4 \exp\left(
    - 2 \frac{(k\sigma/2)^2}{k}
    \right)\\
    &\leq \epsilon,
\end{align*}
if
\begin{align*}
    k \geq  \frac{2}{\sigma^2} \log \frac{n^4}{\epsilon},
\end{align*}
and since $\epsilon<1$ this inequality holds whenever
\begin{align*}
    k \geq  \frac{8}{\sigma^2} \log \frac{n}{\epsilon}.
\end{align*}
That proves the claim.
\end{proof}

\subsection{The $K$ event}

Using the notation of Section~\ref{section:first-step}, fix a quartet of species $\mathcal{Q} = \{A,B,C,D\}$,
let $\P'$ denote the conditional probability given the events $\{\mathcal{X}_\mathcal{Q} \geq \vec{1}\}$, and define $\delta' = \P'[Q_1] - 1/3$.  Since $\P'[Q_2] = \P'[Q_3]$, we have \begin{align}
\label{eq:k-deltap}
    \delta' = \P'[Q_1] - \frac{\P'[Q_1] + \P'[Q_2] + \P'[Q_3]}{3} = \frac{2}{3}\left(\P'[Q_1] - \P'[Q_2]\right)
\end{align}
We seek to bound the right-hand side.

Assume first that the species tree restricted to $\mathcal{Q}$ is balanced. Letting $\P'_i$ indicate $\P'$ conditioned on $\{I=i\}$,
by the proof of Proposition~\ref{PropClaimBal} we have
\begin{align*}
    \P'_i[Q_1] - \P'_i[Q_2] 
    &\geq (\P'_i[Q_1|K] - \P'_i[Q_2|K]) \P'_i[K]
\end{align*}
where we took expectations over $\mathcal{X}_\mathcal{Q}$. Because on the event $\nocoal$, $Q_1$ and $Q_2$ are equally likely by symmetry, we are left with 
\begin{align*}
    \P'_i[Q_1] - \P'_i[Q_2] 
    &\geq \P'_i[K \cap \nocoal^c].
\end{align*}
To bound the right-hand side, 
we consider the event $\mathcal{C}_{ab}$ that the lineages picked from $A$ and $B$ coalesce below $R$.
Notice in particular that $\mathcal{C}_{ab}$ implies $\{i_a = i_b\}$. The event
$K \cap \nocoal^c$ is implied by $\mathcal{C}_{ab}$ together with $\{i_c = i_d = i_a\}$, which are conditionally independent. By Lemma~\ref{LemmaLeg1Copies}, the latter has probability at least $\P'_i[i_c = i_d] \frac{1}{i} \geq \frac{1}{i^2}$. Hence,
\begin{align}
\label{eq:k-cabBal}
    \P'_i[Q_1] - \P'_i[Q_2] 
    &\geq  \P'_i[\mathcal{C}_{ab}]\, \frac{1}{i^2}.
\end{align}

By a similar argument in the
caterpillar case, we have
\begin{align}
\label{eq:k-cabCat}
    \P'_i[Q_1] - \P'_i[Q_2] 
    &\geq \P'_i[K \cap \nocoal^c]\nonumber\\
    &\geq \P'_i[\mathcal{C}_{ab}]\, \frac{1}{i}\nonumber\\    
    &\geq \P'_i[\mathcal{C}_{ab}]\, \frac{1}{i^2}.
\end{align}

It remains to bound $\P'_i[\mathcal{C}_{ab}]$.
\begin{lemma}
\label{lemma:cab}
We have
\begin{align*}
    \P'_i[\mathcal{C}_{ab}]
    \geq \left\{
    \gamma \land \frac{1}{8}
    \right\}
    \frac{1}{i}.
\end{align*}
\end{lemma}
\begin{proof}
Similarly to the proof of Lemma~\ref{LemmaLeg1Copies}, for copy $\ell$ at $R$, let $N_\ell$ be the number of its descendant copies at $R'$, the most recent common ancestor of $A$ and $B$, and let $J = \sum_{\ell=1}^i N_\ell$. We consider two cases for $N_\ell$:
\begin{enumerate}
    \item In the case $N_\ell = 1$ and $\{i_a = i_b = \ell\}$, let $Y_{\ell}$ be the indicator function of the event that the lineages from $a$ and $b$ coalesce before $R$.  So $Y_\ell = 1$ if the lineages from $a$ and $b$ coalesce before $R$, and $0$ otherwise. Under the standard coalescent, the probability of that coalescent event is at least $\gamma$. Here, we are working under the bounded coalescent. In the case that a daughter edge is ancestral to the lineages from $a$ and $b$, the additional conditioning on complete coalescence only increases the probability that $a$ and $b$ coalesce before $R$. So $\gamma$ remains a lower bound.
    
    \item In the case $N_\ell \geq 2$ and $\{i_a = i_b = \ell\}$, let $Z_{\ell}$ be the indicator function of the event that the lineages from $a$ and $b$ coalesce before $R$.  So $Z_\ell = 1$ if the lineages from $a$ and $b$ coalesce before $R$, and $0$ otherwise. Since  $N_\ell \geq 2$, there is at least one duplication below $\ell$ before $R'$. By symmetry, there is probability at least $1/2$ that the first duplication produces a daughter edge with at least half of the descendants of $\ell$ below it at $R'$. Under $\{i_a = i_b = \ell\}$, there is then a probability at least $1/4$ that $a$ and $b$ descend from copies at $R'$ below that daughter edge. So overall there is probability at least $1/8$ that $Z_\ell = 1$ in that case.
\end{enumerate}

Putting these two cases together , we get
\begin{align*}
    \P'_i[\mathcal{C}_{ab}]
    &= \E'_i\left[\E'_i\left[
    \sum_{\ell : N_\ell = 1} \left(\frac{N_\ell}{J}\right)^2 Y_\ell
    + 
    \sum_{\ell : N_\ell > 1} \left(\frac{N_\ell}{J}\right)^2 Z_\ell
    \,\middle|\,(N_\ell)_{\ell=1}^i\right]\right]\\
    &= \E'_i\left[
    \sum_{\ell : N_\ell = 1} \left(\frac{N_\ell}{J}\right)^2 \E'_i\left[Y_\ell\,\middle|\,N_\ell\right]
    + 
    \sum_{\ell : N_\ell > 1} \left(\frac{N_\ell}{J}\right)^2 \E'_i\left[Z_\ell\,\middle|\,N_\ell\right]
    \right]\\
    &\geq \left\{
    \gamma \land \frac{1}{8}
    \right\} \E'_i\left[
    \sum_{\ell : N_\ell = 1} \left(\frac{N_\ell}{J}\right)^2 
    + 
    \sum_{\ell : N_\ell > 1} \left(\frac{N_\ell}{J}\right)^2 
    \right]\\
    &\geq \left\{
    \gamma \land \frac{1}{8}
    \right\} \frac{1}{i},
\end{align*}
as in \cite[Lemma 1]{Legried821439}, proving the claim.
\end{proof}
\begin{lemma}
\label{lemma:deltap}
We have
\begin{align*}
    \delta'
    \geq
    \frac{2}{3}
    \left\{
    \gamma \land \frac{1}{8}
    \right\}
    \frac{\sigma^3}{\alpha^3}.
\end{align*}
\end{lemma}
\begin{proof}
By~\eqref{eq:k-deltap},~\eqref{eq:k-cabBal},~\eqref{eq:k-cabCat}, and Lemma~\ref{lemma:cab},
\begin{align}
\delta'
&= \frac{2}{3}\left(\P'[Q_1] - \P'[Q_2]\right)\nonumber\\
&\geq \frac{2}{3}\left\{
    \gamma \land \frac{1}{8}
    \right\} \E'\left[\frac{1}{I^3}\right]\nonumber\\
    &\geq \frac{2}{3}\left\{
    \gamma \land \frac{1}{8}
    \right\} \frac{1}{\E'\left[I\right]^3},\label{eq:k-jensen}
\end{align}
where the last line follows from Jensen's inequality. Moreover
\begin{align*}
    \alpha &\geq \E[I]\\
    &= \E\left[I\middle|\mathcal{X}_\mathcal{Q} \geq \vec{1}\right] \,\P\left[\mathcal{X}_\mathcal{Q} \geq \vec{1}\right]
    + \E\left[I\middle|\mathcal{X}_\mathcal{Q} < \vec{1}\right] \,\P\left[\mathcal{X}_\mathcal{Q} < \vec{1}\right]\\
    &\geq \E'[I]\,\sigma.
\end{align*}
Plugging back into~\eqref{eq:k-jensen} gives the claim.
\end{proof}

\subsection{Final analysis}

\begin{proof}[Proof of Theorem~\ref{TheoremSampleONE}]
The following, adapted from \cite[Lemmas A.1 and A.2]{Shekhar_2018}, gives a bound on the $k^*$ required to reconstruct the correct species tree with probability $1-\epsilon$ in terms of $\delta'$
\begin{align*}
    k^* > 2 \log\left(\frac{n^4}{\epsilon}\right) \frac{1}{(\delta')^2}.
\end{align*}
In particular, for $\epsilon <1$ this inequality holds whenever 
\begin{align*}
    k^* > 8 \log\left(\frac{n}{\epsilon}\right) \frac{1}{(\delta')^2}.
\end{align*}
By Lemmas~\ref{lemma:kstar}
and~\ref{lemma:deltap}, it suffices to have
\begin{align*}
    k 
    &\geq \left\{\frac{16}{\sigma} \log\left(\frac{n}{\epsilon}\right) \frac{1}{(\frac{2}{3}
    \left\{
    \gamma \land \frac{1}{8}
    \right\}
    \frac{\sigma^3}{\alpha^3})^2} \right\}
    \lor
    \left\{\frac{8}{\sigma^2} \log \frac{n}{\epsilon}\right\}\\
    &\geq \frac{2304 \alpha^6}{\sigma^7 \gamma^2} \log \frac{n}{\epsilon}.
\end{align*}
The claim follows from Lemmas~\ref{lemma:sigma-bound} and~\ref{lemma:alpha-bound}.
\end{proof}

\section{Concluding remarks}

Through a probabilistic analysis of the DLCoal model, we established identifiability of the model species tree and statistical consistency of quartet-based species tree estimation methods ASTRAL-one and ASTRAL-multi.  In our main new result, we derived an upper bound on the required number of gene trees to reconstruct the species tree with high probability.
In particular, we highlighted the roles of the branching process parameters $\lambda$ and $\mu$ as well as the tree depth $\Delta$.  These parameters enter naturally through two relevant quantities: the minimum survival probability of a quartet ($\sigma$) and the maximum expected number of lineages at a vertex ($\alpha$). 

Our results suggest many open problems.  First, can we derive a lower bound (and matching upper bound) on the required number of gene trees for ASTRAL-one and similar methods (including ASTRAL-multi)?  In particular, is our dependence on $\alpha$ (in the supercritical regime) and $\sigma$ (in the subcritical regime) optimal? An improvement on the polynomial dependence on $\alpha$ (see Lemma~\ref{lemma:deltap}) is likely possible with a more detailed analysis of the events in Section~\ref{section:first-step}.  But, perhaps more importantly, are there alternative ways of processing multi-labeled gene trees (not necessarily quartet-based) that dampen or even exclude the effect of $\alpha$?

More generally, it would be interesting to obtain statistical consistency and sample complexity results for models also including LGT, under which at low enough rates quartet-based methods have also been shown to be consistent~\cite{roch12lateral}.  One such more general model was recently introduced in \cite{Li+:20}.

\section*{Acknowledgments}

SR was supported by NSF grants DMS-1614242, CCF-1740707 (TRIPODS), DMS-1902892, and DMS-1916378, as well as a Simons Fellowship and a Vilas Associates Award. BL was supported by NSF grants DMS-1614242, CCF-1740707 (TRIPODS), DMS-1902892 and a Vilas Associates Award (to SR).
MB was supported by NSF grant DMS-1902892 and a Vilas Associates Award (to SR).

\bibliographystyle{alpha}
\bibliography{GDLILS}

\end{document}